\documentclass[11 pt,a4paper]{article}
\usepackage{amsmath, amsthm,amssymb,enumerate,mathrsfs}
\usepackage{tikz}
\usetikzlibrary{calc,shapes}
\usepackage{hyperref}
\usepackage{enumitem}
\usepackage{cite}
\usepackage{breqn}
\usepackage[margin=1.1in]{geometry}

\usepackage[multiple]{footmisc}
\makeatletter
\renewcommand*\@fnsymbol[1]{\the#1}  
\makeatother

\title{Extremal Bounds for Bootstrap Percolation in the Hypercube}
\author{Natasha Morrison\thanks{Department of Pure Mathematics and Mathematical Statistics, University of
Cambridge, Wilberforce Road, Cambridge CB3 0WB, UK. E-mail: \texttt{morrison@dpmms.cam.ac.uk}.}$^{\text{ ,}}$\footnotemark[3]
\and 
Jonathan A. Noel\thanks{Department of Computer Science and DIMAP, University of Warwick, Coventry CV4 7AL,
UK. E-mail: \texttt{j.noel@warwick.ac.uk}.}$^{\text{ ,}}$\thanks{This work was completed while both authors were DPhil students at the University of Oxford.}$^{\text{ ,}}$\thanks{Work of the second author was supported in part by a Postgraduate Scholarship from the Natural Sciences and Engineering Research Council of Canada.}}

\newtheoremstyle{case}{}{}{\normalfont}{}{\itshape}{:}{ }{}

\newtheorem{thm}[equation]{Theorem}
\newtheorem{lem}[equation]{Lemma}

\newtheorem{conj}[equation]{Conjecture}

\newtheorem{ques}[equation]{Question}
\newtheorem{prob}[equation]{Problem}

\theoremstyle{definition}
\newtheorem{defn}[equation]{Definition}

\newtheorem*{ack}{Acknowledgements}

\newtheorem{noteAdded}{Note Added in the Proof}

\newtheoremstyle{case}{}{}{\normalfont}{}{\itshape}{\normalfont:}{ }{}

\theoremstyle{case}
\newtheorem{case}{Case}

\newtheorem{case3}{Case}

\numberwithin{equation}{section}

\DeclareMathOperator{\wsat}{wsat}
\DeclareMathOperator{\vspan}{span}
\DeclareMathOperator{\vdim}{dim}
\DeclareMathOperator{\supp}{supp}

\date{}

\begin{document}

\maketitle

\begin{abstract}
The \emph{$r$-neighbour bootstrap percolation process} on a graph $G$ starts with an initial set $A_0$ of ``infected'' vertices and, at each step of the process, a healthy vertex becomes infected if it has at least $r$ infected neighbours (once a vertex becomes infected, it remains infected forever). If every vertex of $G$ eventually becomes infected, then we say that $A_0$ \emph{percolates}. 

We prove a conjecture of Balogh and Bollob\'{a}s which says that, for fixed $r$ and $d\to\infty$, every percolating set in the $d$-dimensional hypercube has cardinality at least $\frac{1+o(1)}{r}\binom{d}{r-1}$. We also prove an analogous result for multidimensional rectangular grids. Our proofs exploit a connection between bootstrap percolation and a related process, known as \emph{weak saturation}. In addition, we improve on the best known upper bound for the minimum size of a percolating set in the hypercube. In particular, when $r=3$, we prove that the minimum cardinality of a percolating set in the $d$-dimensional hypercube is $\left\lceil\frac{d(d+3)}{6}\right\rceil+1$ for all $d\geq3$. 
\end{abstract}

\section{Introduction}

Given a positive integer $r$ and a graph $G$, the \emph{$r$-neighbour bootstrap percolation process} begins with an initial set of ``infected'' vertices of $G$ and, at each step of the process, a vertex becomes infected if it has at least $r$ infected neighbours. More formally, if $A_0$ is the initial set of infected vertices, then the set of vertices that are infected after the $j$th step of the process for $j\geq1$ is defined by
\[A_j:=A_{j-1}\cup\left\{v\in V(G): \left|N_G(v)\cap A_{j-1}\right|\geq r\right\},\]
where $N_G(v)$ denotes the neighbourhood of $v$ in $G$. We say that $A_0$ \emph{percolates} if $\bigcup_{j=0}^\infty A_j=V(G)$. Bootstrap percolation was introduced by Chalupa, Leath and Reich~\cite{Chalupa} as a mathematical simplification of existing dynamic models of ferromagnetism, but it has also found applications in the study of other physical phenomena such as crack formation and hydrogen mixtures (see Adler and Lev~\cite{AdlerLev}). In addition, advances in bootstrap percolation have been highly influential in the study of more complex processes including, for example, the Glauber dynamics of the Ising model~\cite{RobIsing}.

The main extremal problem in bootstrap percolation is to determine the minimum cardinality of a set which percolates under the $r$-neighbour bootstrap percolation process on $G$; we denote this by $m(G,r)$. An important case is when $G$ is the \emph{$d$-dimensional hypercube} $Q_d$; i.e., the graph with vertex set $\{0,1\}^d$ in which two vertices are adjacent if they differ in exactly one coordinate. Balogh and Bollob\'{a}s~\cite{conj} (see also~\cite{Highdim, LinAlg}) made the following conjecture. 

\begin{conj}[Balogh and Bollob\'{a}s~\cite{conj}] 
\label{theConj}
For fixed $r\geq3$ and $d\to\infty$,
\[m(Q_d,r) = \frac{1+o(1)}{r}\binom{d}{r-1}.\]
\end{conj}

The upper bound of Conjecture~\ref{theConj} is not difficult to prove. Simply let $A_0$ consist of all vertices on ``level $r-2$'' of $Q_d$ and an approximate Steiner system on level $r$, whose existence is guaranteed by an important theorem of R\"{o}dl~\cite{Rodl}; see Balogh, Bollob\'{a}s and Morris~\cite{Highdim} for more details. Note that, under certain conditions on $d$ and $r$, the approximate Steiner system in this construction can be replaced with an exact Steiner system (using, for example, the celebrated result of Keevash~\cite{Keevash}). In this special case, the percolating set has cardinality $\frac{1}{r}\binom{d}{r-1} + \binom{d}{r-2}$ which yields
\begin{equation}
\label{Steiner}
m\left(Q_d,r\right) \leq  \frac{d^{r-1}}{r!} + \frac{d^{r-2}(r+2)}{2r(r-2)!}+O\left(d^{r-3}\right).
\end{equation}
Lower bounds have been far more elusive; previously, the best known lower bound on $m\left(Q_d,r\right)$ for fixed $r\geq3$ was only linear in $d$~\cite{Highdim}. In this paper, we prove Conjecture~\ref{theConj}. 

\begin{thm}
\label{hyper}
For $d\geq r\geq1$, 
\[m\left(Q_d,r\right)\geq 2^{r-1} + \sum_{j=1}^{r-1}\binom{d-j-1}{r-j}\frac{j2^{j-1}}{r}\]
where, by convention, $\binom{a}{b}=0$ when $a<b$. 
\end{thm}

For fixed $r\geq3$, Theorem~\ref{hyper} implies
\[m(Q_d,r)\geq \frac{d^{r-1}}{r!} + \frac{d^{r-2}(6-r)}{2r(r-2)!}+\Omega\left(d^{r-3}\right),\]
which differs from the upper bound in \eqref{Steiner} by an additive term of order $\Theta\left(d^{r-2}\right)$. We will also provide a recursive upper bound on $m\left(Q_d,r\right)$, which improves on the second order term of \eqref{Steiner}. For $r=3$, we combine this recursive bound with some additional arguments to show that Theorem~\ref{hyper} is tight in this case. 

\begin{thm}
\label{r=3Upper}
For $d\geq3$, we have $m(Q_d,3)=\left\lceil\frac{d(d+3)}{6}\right\rceil+1$.
\end{thm}

In order to prove Theorem~\ref{hyper}, we will exploit a relationship between bootstrap percolation and the notion of weak saturation introduced by Bollob\'{a}s~\cite{wsat}. Given fixed graphs $G$ and $H$, we say that a spanning subgraph $F$ of $G$ is \emph{weakly $(G,H)$-saturated} if the edges of $E(G)\setminus E(F)$ can be added to $F$, one edge at a time, in such a way that each edge completes a copy of $H$ when it is added. The main extremal problem in weak saturation is to determine the \emph{weak saturation number} of $H$ in $G$ defined by
\[\wsat(G,H):=\min\left\{|E(F)|: F\text{ is weakly $(G,H)$-saturated}\right\}.\]  
Weak saturation is very well studied (see, e.g.~\cite{Alon,Kalai1,Kalai2, MNS,MoshShap,Oleg}). Our proof of Theorem~\ref{hyper} relies on the following bound, which is easy to prove:
\begin{equation}\label{wsatperc}
m(G,r) \geq \frac{\wsat\left(G,S_{r+1}\right)}{r}
\end{equation}
where $S_{r+1}$ denotes the star with $r+1$ leaves. A slightly stronger version of \eqref{wsatperc} is stated and proved in the next section. We obtain an exact expression for the weak saturation number of $S_{r+1}$ in the $d$-dimensional hypercube $Q_d$. 

\begin{thm}
\label{wsatcube}
If $d \ge r \ge 0$, then
\[\wsat\left(Q_d, S_{r+1}\right) = r2^{r-1} + \sum_{j=1}^{r-1} \binom{d-j-1}{r-j}j 2^{j-1}.\]
\end{thm}

Theorem~\ref{hyper} follows directly from this theorem and \eqref{wsatperc}. We also determine an exact expression for the weak saturation number of $S_{r+1}$ in the $d$-dimensional $a_1\times\cdots \times a_d$ grid, denoted by $\prod_{i=1}^d[a_i]$. We state this result here in the case $d\geq r$; a more general result is expressed later in terms of a recurrence relation.

\begin{thm}
\label{genwsat}
For $d\geq r\geq1$ and $a_1,\dots,a_d\geq2$, 
\begin{dmath*}{\wsat\left(\prod_{i=1}^d[a_i],S_{r+1}\right)} = \sum_{\substack{S\subseteq [d]\\ |S|\leq r-1}}\left(\prod_{i\in S}(a_i-2)\right)\left({(r-|S|)2^{r-|S|-1}} + {\sum_{j=1}^{r-|S|-1}\binom{d-|S|-j-1}{r-|S|-j}j2^{j-1}}\right).\end{dmath*}
\end{thm}

Observe that a lower bound on $m\left(\prod_{i=1}^d[a_i],r\right)$ follows from Theorem~\ref{genwsat} and \eqref{wsatperc}. To our knowledge, the combination of Theorem~\ref{genwsat} and \eqref{wsatperc} implies all of the known lower bounds on the cardinality of percolating sets in  multidimensional grids. In particular, it implies the (tight) lower bounds
\[\label{d=r}m\left([n]^d,d\right)\geq n^{d-1},\]
and
\begin{equation}\label{r=2}m\left(\prod_{i=1}^d[a_i],2\right) \geq \left\lceil\frac{\sum_{i=1}^d(a_i-1)}{2}\right\rceil + 1.\end{equation}
established in~\cite{Pete} and~\cite{conj}, respectively. 

Some motivation for Conjecture~\ref{theConj} stems from its connection to problems of a probabilistic nature. The most well studied problem in bootstrap percolation is to estimate the \emph{critical probability} at which a randomly generated set of vertices in a graph $G$ becomes likely to percolate. To be more precise, for $p\in [0,1]$, suppose that $A_0^p$ is a subset of $V(G)$ obtained by including each vertex randomly with probability $p$ independently of all other vertices and define
\[p_c(G,r):=\inf\left\{p: \mathbb{P}\left(A_0^p \text{ percolates}\right)\geq 1/2 \right\}.\]
The problem of estimating $p_c\left([n]^d,r\right)$ for fixed $d$ and $r$ and $n\to\infty$ was first considered by Aizenman and Lebowitz~\cite{Aiz} and subsequently studied in~\cite{3Dim,Cerf,Cerf2,sharper,Holroyd}. This rewarding line of research culminated in a paper of Balogh, Bollob{\'a}s, Duminil-Copin and Morris~\cite{sharp} in which $p_c\left([n]^d,r\right)$ is determined asymptotically for all fixed values of $d$ and $2\leq r\leq d$.

Comparably, far less is known about the critical probability when $d$ tends to infinity. In this regime, the main results are due to Balogh, Bollob\'{a}s and Morris in the case $r=d$~\cite{majority} and $r=2$~\cite{Highdim}. In the latter paper, the extremal bound \eqref{r=2} was applied to obtain precise asymptotics for $p_c\left([n]^d,2\right)$ whenever $d\gg \log(n)\geq 1$. In contrast, very little is known about the critical probability for fixed $r\geq3$ and $d\to\infty$. For example, the logarithm of $p_c\left(Q_d,3\right)$ is not even known to within a constant factor (see~\cite{Highdim}). As was mentioned in~\cite{LinAlg}, the original motivation behind Conjecture~\ref{theConj} was in its connections to the problem of estimating $p_c(Q_d,r)$.

The rest of the paper is organised as follows. In the next section, we outline our approach to proving Theorems~\ref{hyper} and~\ref{genwsat} and establish some preliminary lemmas. In Section~\ref{hyperSec}, we determine the value of $\wsat\left(Q_d,S_{r+1}\right)$, from which we derive Theorem~\ref{hyper} via \eqref{wsatperc}. We then determine $\wsat\left(\prod_{i=1}^d[a_i],S_{r+1}\right)$ in full generality in Section~\ref{generalSec} using similar ideas (which become somewhat more cumbersome in the general setting). In Section~\ref{upperSec}, we provide constructions of small percolating sets in the hypercube  and prove Theorem~\ref{r=3Upper}. Finally, we conclude the paper in Section~\ref{concl} by stating some open problems related to our work.

\section{Preliminaries}
\label{prelim}

We open this section by proving the following lemma, which is slightly stronger than \eqref{wsatperc} when applied to graphs with vertices of degree less than $r$ (including, for example, the graph $\prod_{i=1}^d[a_i]$ for $d<r$). Given a graph $F$ and a vertex $v\in V(F)$, the \emph{degree} of $v$ in $F$, denoted $d_F(v)$, is the number of vertices of $F$ which are adjacent to $v$. That is,  $d_F(v):=\left|N_F(v)\right|$. 

\begin{lem}
\label{edge}
Let $G$ be a graph and let $F$ be a spanning subgraph of $G$. Define
\[A_F:=\left\{v\in V(G): d_F(v)\geq\min\left\{r,d_G(v)\right\}\right\}.\]
If $A_F$ percolates with respect to the $r$-neighbour bootstrap percolation process on $G$, then $F$ is weakly $\left(G,S_{r+1}\right)$-saturated.  
\end{lem}

\begin{proof}
Let $n:=|V(G)|$. By hypothesis, we can label the vertices of $G$ by $v_1,\dots,v_{n}$ in such a way that
\begin{itemize}
\item $\left\{v_1,\dots,v_{\left|A_F\right|}\right\}= A_F$, and
\item for $|A_F|+1\leq i\leq n$, the vertex $v_i$ has at least $r$ neighbours in $\left\{v_1,\dots,v_{i-1}\right\}$. 
\end{itemize}
Let us show that $F$ is weakly $(G,S_{r+1})$-saturated. We begin by adding to $F$ every edge of $E(G)\setminus E(F)$ which is incident to a vertex in $A_F$ (one edge at a time in an arbitrary order). For every vertex $v\in A_F$, we have that either
\begin{itemize}
\item there are at least $r$ edges of $F$ incident to $v$, or
\item every edge of $G$ incident with $v$ is already present in $F$.
\end{itemize}
Therefore, every edge of $E(G)\setminus E(F)$  incident to a vertex in $A_F$ completes a copy of $S_{r+1}$ when it is added. 

Now, for each $i=|A_F|+1,\dots, n$ in turn, we add every edge incident to $v_i$ which has not already been added (one edge at a time in an arbitrary order). Since $v_i$ has at least $r$ neighbours in $\left\{v_1,\dots,v_{i-1}\right\}$ and every edge incident to a vertex in $\left\{v_1,\dots,v_{i-1}\right\}$ is already present, we get that every such edge completes a copy of $S_{r+1}$ when it is added. The result follows. 
\end{proof}

For completeness, we will now deduce \eqref{wsatperc} from Lemma~\ref{edge}.

\begin{proof}[Proof of \eqref{wsatperc}]
Let $A_0$ be a set of cardinality $m(G,r)$ which percolates with respect to the $r$-neighbour bootstrap percolation process on $G$ and let $F$ be a spanning subgraph of $G$ such that  $d_F(v)\geq \min\left\{d_G(v),r\right\}$ for each $v\in A_0$. Note that this can be achieved by including at most $r$ edges per vertex of $A_0$ and so we can assume that $|E(F)|\leq r|A_0|=rm(G,r)$. By Lemma~\ref{edge}, $F$ is weakly $(G,S_{r+1})$-saturated and so
\[\wsat\left(G,S_{r+1}\right)\leq |E(F)|\leq rm(G,r)\]
as required. 
\end{proof}

We turn our attention to determining the weak saturation number of stars in hypercubes and, more generally, in multidimensional rectangular grids. To prove an upper bound on a weak saturation number, one only needs to construct a \emph{single} example of a weakly saturated graph of small size. Our main tool for proving the lower bound is the following linear algebraic lemma of Balogh, Bollob\'{a}s, Morris and Riordan~\cite{LinAlg}. A major advantage of this lemma is that it allows us to prove the lower bound in a constructive manner as well.  We include a proof for completeness.

\begin{lem}[Balogh, Bollob\'{a}s, Morris and Riordan~\cite{LinAlg}]
\label{linalg}
Let $G$ and $H$ be graphs and let $W$ be a vector space. Suppose that $\left\{f_e: e\in E(G)\right\}$ is a collection of vectors in $W$ such that for every copy $H'$ of $H$ in $G$ there exists non-zero coefficients $\left\{c_e:e\in E(H')\right\}$ such that $\sum_{e\in E(H')}c_ef_e = 0$. Then
\[\wsat(G,H) \geq \vdim(\vspan\left\{f_e:e\in E(G)\right\}).\]
\end{lem}

\begin{proof}
Let $F$ be a weakly $(G,H)$-saturated graph and define $m:=|E(G)\setminus E(F)|$. By definition of $F$, we can label the edges of $E(G)\setminus E(F)$ by $e_1,\dots,e_m$ in such a way that, for $1\leq i\leq m$, there is a copy $H_i$ of $H$ in $F_i:=F\cup\left\{e_1,\dots,e_i\right\}$ containing the edge $e_i$. By the hypothesis, we get that
\[f_{e_i}\in \vspan\left\{f_e: e\in E(H_i)\setminus\left\{e_i\right\}\right\}\subseteq \vspan\left\{f_e: e\in E(F_i)\setminus\left\{e_i\right\}\right\}\]
for all $i$. Therefore,
\[|E(F)|\geq \vdim\left(\vspan\left\{f_e: e\in E(F)\right\}\right) = \vdim\left(\vspan\left\{f_e: e\in E\left(F_1\right)\right\}\right)\]
\[=\dots =\vdim\left(\vspan\left\{f_e:e\in E\left(F_m\right)\right\}\right) = \vdim\left(\vspan\left\{f_e:e\in E(G)\right\}\right).\] 
The result follows. 
\end{proof}

Lemma~\ref{linalg} was proved in a more general form and applied to a percolation problem in multidimensional square grids in~\cite{LinAlg}. It was also used by Morrison, Noel and Scott~\cite{MNS} to determine $\wsat\left(Q_d,Q_m\right)$ for all $d\geq m\geq1$. We remark that the general idea of applying the notions of dependence and independence in weak saturation problems is also present in the works of Alon~\cite{Alon} and Kalai~\cite{Kalai2}, where techniques involving exterior algebra and matroid theory were used to prove a tight lower bound on $\wsat(K_n,K_k)$ conjectured by Bollob\'{a}s~\cite{Bollobook}. For a more recent application of exterior algebra and matroid theory to weak saturation problems, see~\cite{Oleg}.

\section{The Hypercube Case}
\label{hyperSec}

Our goal in this section is to prove Theorem~\ref{wsatcube}. This settles the case $a_1=\dots=a_d=2$ of Theorem~\ref{genwsat} and, as discussed earlier, implies Theorem~\ref{hyper} via \eqref{wsatperc}. First, we require some definitions. 

\begin{defn}
Given $k\geq 1$, an index $i\in [k]$ and  $x\in \mathbb{R}^k$, let $x_i$ denote the $i$th coordinate of $x$. The \emph{support} of $x$ is defined by $\supp(x):=\left\{i\in[k]: x_i\neq 0\right\}$. 
\end{defn}

\begin{defn}
The \emph{direction} of an edge $e=uv\in E\left(Q_d\right)$ is the unique index $i\in [d]$ such that $u_i\neq v_i$. Given a vertex $v\in V(Q_d)$, we define $e(v,i)$ to be the unique edge in direction $i$ that is incident to $v$. 
\end{defn}

Note that each edge of $Q_d$ receives two labels (one for each of its endpoints). Our approach will make use of the following simple linear algebraic fact.

\begin{lem}
\label{support}
Let $k\geq \ell\geq0$ be fixed. Then there exists a subspace $X$ of $\mathbb{R}^k$ of dimension $k-\ell$ such that $\left|\supp(x)\right|\geq \ell+1$ for every $x\in X\setminus \{0\}$. 
\end{lem}

\begin{proof}
Define $X$ to be the span of a set $\left\{v_1,\dots,v_{k-\ell}\right\}$ of unit vectors of $\mathbb{R}^k$ chosen independently and uniformly at random with respect to the standard measure on the unit sphere $S^{k-1}$. Given a fixed subspace $W$ of $\mathbb{R}^k$ of dimension at most $\ell$ and $1\leq i\leq k-\ell$, the space
\[\vspan\left(W\cup\left\{v_1,\dots,v_{i-1}\right\}\right)\]
has dimension less than $k$. Thus, the unit sphere of this space has measure zero in $S^{k-1}$ and so, with probability one, $v_{i}\notin\vspan\left(W\cup\left\{v_1,\dots,v_{i-1}\right\}\right)$. It follows that $\vdim(X)=k-\ell$ and $X\cap W=\{0\}$ almost surely. In particular, if we let $T\subseteq [k]$ be a fixed set of cardinality $\ell$ and define
\[W_T:=\left\{x\in\mathbb{R}^k: \supp(x)\subseteq T\right\},\]
then $X\cap W_T=\{0\}$ almost surely. Since there are only finitely many sets $T\subseteq [k]$ of cardinality $\ell$, we can assume that $X$ is chosen so that $X\cap W_T=\{0\}$ for every such set. This completes the proof. 
\end{proof}

In the appendix, we provide an explicit (i.e. non-probabilistic) example of a vector space $X$ satisfying Lemma~\ref{support}. The following lemma highlights an important property of the space $X$. 

\begin{lem}
\label{Tsupp}
Let $k\geq\ell\geq0$ and let $X$ be a subspace of $\mathbb{R}^k$ of dimension $k-\ell$ such that $|\supp(x)| \ge \ell + 1$ for every $x \in X\backslash \{0\}$. For every set $T \subseteq [k]$ of cardinality $\ell+1$, there exists $x \in X$ with $\supp(x) = T$.
\end{lem}

\begin{proof}
Let $T\subseteq[k]$ with $|T|=\ell+1$. Clearly, the space $\{x\in \mathbb{R}^k: \supp(x)\subseteq T\}$ has dimension $\ell+1$. Therefore, since $\vdim(X)=k-\ell$, there must be a non-zero vector $x\in X$ with $\supp(x)\subseteq T$. However, this inclusion must be equality since $|\supp(x)|\geq \ell+1$. 
\end{proof}

We are now in position to prove Theorem~\ref{wsatcube}. For notational convenience, we write 
\[w(r,d) :=r2^{r-1} + \sum_{j=1}^{r-1} \binom{d-j-1}{r-j}j 2^{j-1}.\]
We deduce Theorem~\ref{wsatcube} from the following lemma, after which we will prove the lemma itself.

\begin{lem}
\label{theProp}
Let $X$ be a subspace of $\mathbb{R}^d$ of dimension $d-r$ such that $|\supp(x)|\geq r+1$ for every $x\in X\setminus\{0\}$. Then there is a spanning subgraph $F$ of $Q_d$ and a collection $\left\{f_e: e\in E\left(Q_d\right)\right\}\subseteq \mathbb{R}^{w(r,d)}$ such that
\begin{enumerate}
\renewcommand{\theenumi}{Q\arabic{enumi}}
\renewcommand{\labelenumi}{(\theenumi)}
{\setlength\itemindent{4pt}\item \label{wsat} $F$ is weakly $\left(Q_d,S_{r+1}\right)$-saturated and $|E(F)| = w(r,d)$,}
{\setlength\itemindent{4pt}\item\label{satis} $\sum_{i=1}^d x_i f_{e(v,i)} = 0$ for every $v\in V\left(Q_d\right)$ and $x\in X$, and}
{\setlength\itemindent{4pt}\item\label{BigDim} $\vspan\left\{f_e: e\in E\left(Q_d\right)\right\}=\mathbb{R}^{w(r,d)}$. }
\end{enumerate}
\end{lem}

\begin{proof}[Proof of Theorem~\ref{wsatcube}]
Clearly, the existence of a graph $F$ satisfying (\ref{wsat}) implies the upper bound $\wsat(Q_d,S_{r+1})\leq w(r,d)$. We show that the lower bound follows from (\ref{satis}), (\ref{BigDim}) and Lemma~\ref{linalg}. Note that the edge sets of copies of $S_{r+1}$ in $Q_d$ are precisely the sets of the form $\{e(v,i): i\in T\}$ where $v$ is a fixed vertex of $Q_d$ and $T$ is a subset of $[d]$ of cardinality $r+1$. By Lemma~\ref{Tsupp} we know that there exists some $x \in X$ with $\supp(x) = T$. By (\ref{satis}) we have 
\[\sum_{i=1}^d x_i f_{e(v,i)} = \sum_{i\in T}x_if_{e(v,i)} = 0\]
and so the hypotheses of Lemma~\ref{linalg} are satisfied. Therefore, 
\[\wsat\left(Q_d,S_{r+1}\right)\geq \vdim(\vspan\left\{f_e:e\in E(Q_d)\right\})\]
which equals $w(r,d)$ by (\ref{BigDim}). The result follows.
\end{proof}

\begin{proof}[Proof of Lemma~\ref{theProp}]
We proceed by induction on $d$. We begin by settling some easy boundary cases before explaining the inductive step. 

\begin{case}
$r=0$. 
\end{case}

In this case, $S_{r+1}$ is isomorphic to $K_2$. Also, $w(r,d)=0$ and $X=\mathbb{R}^d$. We let $F$ be a spanning subgraph of $Q_d$ with no edges and set $f_e:=0$ for every $e\in Q_d$. It is trivial to check that (\ref{wsat}), (\ref{satis}) and (\ref{BigDim}) are satisfied. 

\begin{case}
$d=r\geq 1$. 
\end{case}

In this case, $w(r,d)=d2^{d-1}= \left|E\left(Q_d\right)\right|$ and $X=\{0\}$. We define $F:=Q_d$ and let $\left\{f_e: e\in E\left(Q_d\right)\right\}$ be a basis for $\mathbb{R}^{w(r,d)}$. Clearly (\ref{wsat}), (\ref{satis}) and (\ref{BigDim}) are satisfied. 

\begin{case}
$d>r\geq1$. 
\end{case}

We begin by constructing $F$ in such a way that (\ref{wsat}) is satisfied. For $i\in \{0,1\}$, let $Q_{d-1}^i$ denote the subgraph of $Q_d$ induced by $\{0,1\}^{d-1}\times\{i\}$. Note that both $Q_{d-1}^0$ and $Q_{d-1}^1$ are isomorphic to $Q_{d-1}$. Let $F$ be a spanning subgraph of $Q_d$ such that
\begin{itemize}
\item the subgraph $F_0$ of $F$ induced by $V\left(Q_{d-1}^0\right)$ is a weakly $(Q_{d-1},S_{r+1})$-saturated graph of minimum size, 
\item the subgraph $F_1$ of $F$ induced by $V\left(Q_{d-1}^1\right)$ is a weakly $(Q_{d-1},S_{r})$-saturated graph of minimum size, and
\item $F$ contains no edge in direction $d$.
\end{itemize}
Figure~\ref{Q5} contains a specific instance of this construction. 

By the inductive hypothesis and our choice of $F_0$ and $F_1$, we know that $|E(F_0)|=w(r,d-1)$ and $|E(F_1)|=w(r-1,d-1)$. So, by construction of $F$ and Pascal's Formula for binomial coefficients, we have
\[|E(F)|=w(r,d-1)+w(r-1,d-1)\]
\[= \left[r2^{r-1} + \sum_{j=1}^{r-1}\binom{(d-1)-j-1}{r-j}j2^{j-1}\right] + \left[(r-1)2^{r-2} + \sum_{j=1}^{r-2}\binom{(d-1)-j-1}{(r-1)-j}j2^{j-1}\right]\]
\[=r2^{r-1} +\left(\sum_{j=1}^{r-2} \binom{d-j-1}{r-j} j2^{j-1}\right)  + (d-r-1)(r-1)2^{r-2} + (r-1)2^{r-2} = w(r,d).\]
Let us verify that $F$ is weakly $\left(Q_d,S_{r+1}\right)$-saturated. To see this, we add the edges of $E\left(Q_d\right)\setminus E(F)$ to $F$ in three stages. By construction, we can begin by adding all edges of $Q_{d-1}^0$ which are not present in $F_0$ in such a way that each edge completes a copy of $S_{r+1}$ in $Q_{d-1}^0$ when it is added. In the second stage, we add all edges of $Q_d$ in direction $d$ one by one in any order. Since every vertex of $Q_d$ has degree $d\geq r+1$ and every edge of $Q_{d-1}^0$ has already been added, we get that every edge added in this stage completes a copy of $S_{r+1}$ in $Q_d$. Finally, we add the edges of $Q_{d-1}^1$ which are not present in $F_1$ in such a way that each added edge completes a copy of $S_{r}$ in $Q_{d-1}^1$. Since the edges in direction $d$ have already been added, we see that every such edge completes a copy of $S_{r+1}$ in $Q_d$. Therefore, (\ref{wsat}) holds.

\begin{figure}[htbp]
\newcommand{\Depth}{1.6}
\newcommand{\Height}{1.6}
\newcommand{\Width}{1.6}
\newcommand{\Shift}{2.5*\Height}
\newcommand{\Sep}{0.17*\Shift}
\tikzstyle{loosely dashed}=          [dash pattern=on 5pt off 3pt]
\begin{center}
\begin{tikzpicture}

\coordinate (O1) at (0,0,0);
\coordinate (A1) at (0,\Width,0);
\coordinate (B1) at (0,\Width,\Height);
\coordinate (C1) at (0,0,\Height);
\coordinate (D1) at (\Depth,0,0);
\coordinate (E1) at (\Depth,\Width,0);
\coordinate (F1) at (\Depth,\Width,\Height);
\coordinate (G1) at (\Depth,0,\Height);

\draw[very thick] (O1) -- (C1) -- (G1) -- (D1) -- cycle;
\draw[very thick,dashed] (O1) -- (A1) -- (E1) -- (D1);
\draw[very thick,dashed] (A1) -- (B1) -- (C1);
\draw[very thick,dashed] (E1) -- (F1) -- (G1);
\draw[very thick]  (B1) -- (F1);

\draw (\Depth/2 +\Shift/2, \Width/2,\Height/2) node [rectangle, draw, thick, rounded corners, minimum height=6.2em,minimum width=18em]{};
\draw (\Depth/2+\Shift/2,0,\Height/2) node [label={[label distance=0.45cm]270:$F_{1}$}] {};

\coordinate (O2) at (0+\Shift,0,0);
\coordinate (A2) at (0+\Shift,\Width,0);
\coordinate (B2) at (0+\Shift,\Width,\Height);
\coordinate (C2) at (0+\Shift,0,\Height);
\coordinate (D2) at (\Depth+\Shift,0,0);
\coordinate (E2) at (\Depth+\Shift,\Width,0);
\coordinate (F2) at (\Depth+\Shift,\Width,\Height);
\coordinate (G2) at (\Depth+\Shift,0,\Height);

\draw[very thick,dashed] (O2) -- (C2) -- (G2) -- (D2) -- cycle;
\draw[very thick,dashed] (O2) -- (A2) -- (E2) -- (D2);
\draw[very thick,dashed] (A2) -- (B2) -- (C2);
\draw[very thick,dashed] (E2) -- (F2) -- (G2);
\draw[very thick]  (B2) -- (F2);

\coordinate (O3) at (0,-\Shift,0);
\coordinate (A3) at (0,\Width-\Shift,0);
\coordinate (B3) at (0,\Width-\Shift,\Height);
\coordinate (C3) at (0,-\Shift,\Height);
\coordinate (D3) at (\Depth,-\Shift,0);
\coordinate (E3) at (\Depth,\Width-\Shift,0);
\coordinate (F3) at (\Depth,\Width-\Shift,\Height);
\coordinate (G3) at (\Depth,-\Shift,\Height);

\draw[very thick] (O3) -- (C3) -- (G3) -- (D3) -- cycle;
\draw[very thick] (O3) -- (A3) -- (E3) -- (D3);
\draw[very thick] (A3) -- (B3) -- (C3);
\draw[very thick] (E3) -- (F3) -- (G3);
\draw[very thick]  (B3) -- (F3);

\draw (\Depth/2 +\Shift/2, \Width/2 - \Shift,\Height/2) node [rectangle, draw, thick, rounded corners, minimum height=6.2em,minimum width=18em]{};
\draw (\Depth/2+\Shift/2,-\Shift,\Height/2) node [label={[label distance=0.45cm]270:$F_{0}$}] {};

\coordinate (O4) at (0+\Shift,0-\Shift,0);
\coordinate (A4) at (0+\Shift,\Width-\Shift,0);
\coordinate (B4) at (0+\Shift,\Width-\Shift,\Height);
\coordinate (C4) at (0+\Shift,0-\Shift,\Height);
\coordinate (D4) at (\Depth+\Shift,0-\Shift,0);
\coordinate (E4) at (\Depth+\Shift,\Width-\Shift,0);
\coordinate (F4) at (\Depth+\Shift,\Width-\Shift,\Height);
\coordinate (G4) at (\Depth+\Shift,0-\Shift,\Height);

\draw[very thick] (O4) -- (C4) -- (G4) -- (D4) -- cycle;
\draw[very thick,dashed] (O4) -- (A4) -- (E4) -- (D4);
\draw[very thick,dashed] (A4) -- (B4) -- (C4);
\draw[very thick,dashed] (E4) -- (F4) -- (G4);
\draw[very thick]  (B4) -- (F4);

\draw[line width=2,loosely dashed] (\Depth +\Sep, \Width/2,\Height/2) -- (\Shift-\Sep, \Width/2,\Height/2);

\draw[line width=2,loosely dashed] (\Depth +\Sep, \Width/2 - \Shift,\Height/2) -- (\Shift-\Sep, \Width/2 - \Shift,\Height/2);

\draw[line width=2,loosely dashed] (\Depth/2,  -\Sep ,\Height/2) -- (\Depth/2, -\Shift+\Sep +\Width,\Height/2);

\draw[line width=2,loosely dashed] (\Shift + \Depth/2,  -\Sep ,\Height/2) -- (\Shift +\Depth/2, -\Shift+\Sep +\Width,\Height/2);

\end{tikzpicture}
\end{center}

\caption{A weakly $(Q_5,S_4)$-saturated graph $F$ constructed inductively from a weakly $(Q_4,S_4)$-saturated graph $F_0$ and a weakly $(Q_4,S_3)$-saturated graph $F_1$, each of which is also constructed inductively. }
\label{Q5}
\end{figure}
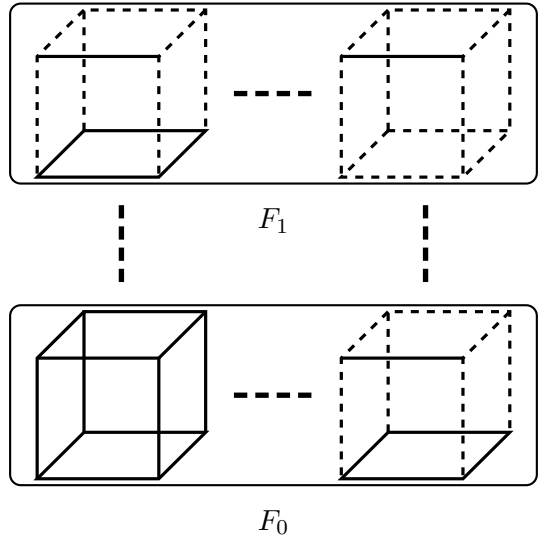

Thus, all that remains is to construct $\left\{f_e: e\in E\left(Q_d\right)\right\}$ in such a way that (\ref{satis}) and (\ref{BigDim}) are satisfied. Let $\pi:X\to \mathbb{R}^{d-1}$ be the standard projection defined by $\pi:\left(x_1,\dots,x_d\right)\mapsto \left(x_1,\dots,x_{d-1}\right)$. Let $z\in X$ be an arbitrary vector such that $d\in\supp\left(z\right)$ (such a vector exists by Lemma~\ref{Tsupp}) and let $T_{z}: X\to X$ be the linear map defined by
\[T_{z}(x):= x - \frac{x_d}{z_d}z\]
for $x\in X$. Define
\[X_0:= \pi\left(T_{z}(X)\right)\text { and}\]
\[X_1:= \pi(X).\]
What we will do next is apply the inductive hypothesis  to assign the edges of $Q_{d-1}^0$ and $Q_{d-1}^1$ to  vectors in $\mathbb{R}^{w(r,d-1)}$ and $\mathbb{R}^{w(r-1,d-1)}$, respectively, satisfying conditions analogous to (\ref{satis}) and (\ref{BigDim}) with the spaces $X_0$ and $X_1$ playing the role of $X$. These vectors will then be used to construct vectors for the edges of $Q_d$. In order to do this, we need to verify that $X_0$ and $X_1$ satisfy the hypotheses of the lemma. 

Clearly, $X_0$ and $X_1$ are both contained in $\mathbb{R}^{d-1}$, simply by definition of the projection $\pi$. The kernel of $T_z$ is precisely the span of $\{z\}$, and so $T_z(X)$ is a subspace of $X$ of dimension $d-r-1$. In particular, since $T_z(X)\subseteq X$, every non-zero  $x\in T_z(X)$ has $|\supp(x)|\geq r+1$. Also, the last coordinate of every $x\in T_z(X)$ is equal to zero, and so $\pi(x)$ has the same support as $x$. Thus, $X_0$ is a space of dimension $d-r-1$ such that every non-zero  $x\in X_0$ has $|\supp(x)|\geq r+1$. Now, for $X_1$, we note that the support of each non-zero vector of $X$ has size at least $r+1\geq2$ and so the kernel of $\pi$ is $\{0\}$. This implies that $X_1$ has dimension $d-r$. The fact that every non-zero $x\in X_1$ has $|\supp(x)|\geq r$ follows easily from the fact that every non-zero vector of $X$ is supported on a set of size at least $r+1$.

Therefore, by the inductive hypothesis, there exists $\left\{f_e^0: e\in E\left(Q_{d-1}^0\right)\right\}$ in $\mathbb{R}^{w(r,d-1)}$ and $\left\{f_e^1: e\in E\left(Q_{d-1}^1\right)\right\}$ in $\mathbb{R}^{w(r-1,d-1)}$ such that
\begin{enumerate}
\renewcommand{\theenumi}{Q2.\arabic{enumi}}
\renewcommand{\labelenumi}{(\theenumi)}
\addtocounter{enumi}{-1}
{\setlength\itemindent{10pt}\item\label{satisfied0} $\sum_{i=1}^{d-1} x_i f_{e(v,i)}^0 = 0$ for every $v\in V\left(Q_{d-1}^0\right)$ and $x\in X_0$,}
{\setlength\itemindent{10pt}\item\label{satisfied1} $\sum_{i=1}^{d-1} x_i f_{e(v,i)}^1 = 0$ for every $v\in V\left(Q_{d-1}^1\right)$ and $x\in X_1$,}
\renewcommand{\theenumi}{Q3.\arabic{enumi}}
\renewcommand{\labelenumi}{(\theenumi)}
\addtocounter{enumi}{-2}
{\setlength\itemindent{10pt}\item\label{bigDim0} $\vspan\left\{f_e^0: e\in E\left(Q_{d-1}^0\right)\right\}=\mathbb{R}^{w(r,d-1)}$, and}
{\setlength\itemindent{10pt}\item\label{bigDim1} $\vspan\left\{f_e^1: e\in E\left(Q_{d-1}^1\right)\right\}=\mathbb{R}^{w(r-1,d-1)}$. }
\end{enumerate}

We will define the vectors $\left\{f_e: e\in E\left(Q_d\right)\right\}\subseteq \mathbb{R}^{w(r,d-1)}\oplus\mathbb{R}^{w(r-1,d-1)} \simeq \mathbb{R}^{w(r,d)}$ satisfying (\ref{satis}) and (\ref{BigDim}) in three stages. First, if $e\in E\left(Q_{d-1}^0\right)$, then we set
\[f_e:=f_e^0\oplus 0.\]
Next, for each edge of the form $e = e(v,d)$ for $v\in V\left(Q_{d-1}^0\right)$, we let 
\begin{equation}
\label{xstar}
f_e:= -\frac{1}{z_d}\sum_{i=1}^{d-1}z_if_{e(v,i)}
\end{equation}
(recall the definition of $z$ above). Finally, if $e=uv\in E\left(Q_{d-1}^1\right)$, then we let $e' = u'v'$ where $u'$ and $v'$ are the unique neighbours of $u$ and $v$ in $V\left(Q_{d-1}^0\right)$ and define
\[f_e:=f_{e'}^0\oplus f_e^1.\]
Let us verify that (\ref{BigDim}) holds. The span of $\left\{f_e: e\in E(Q_d)\right\}$ contains the vectors
\[\left\{f_e^0\oplus 0: e\in E\left(Q_{d-1}^0\right)\right\}\]
as well as the vectors 
\[\left\{0\oplus f_e^1: e\in E\left(Q_{d-1}^1\right)\right\}\]
since, for every $e\in E\left(Q_{d-1}^1\right)$, we have $\left(f_{e'}^0\oplus f_e^1\right) - \left(f_{e'}^0\oplus 0\right) = 0\oplus f_e^1$. By (\ref{GbigDim0}) and (\ref{GbigDim1}), these sets span spaces of dimensions $w(r,d-1)$ and $w(r-1,d-1)$, respectively. Also, all of the vectors in the first space are clearly orthogonal to those in the second space. Thus, $\{f_e: e\in E(Q_d)\}$ spans a space of dimension $w(r,d-1)+w(r-1,d-1)=w(r,d)$ and so (\ref{BigDim}) holds.

Finally, we prove that (\ref{satis}) is satisfied. First, let $v\in V\left(Q_{d-1}^0\right)$ and let $x\in X$ be arbitrary. Define $x^\dagger:=T_{z}(x)$ and note that $d\notin \supp\left(x^\dagger\right)$ by definition of $T_z$ and that $x=x^\dagger + \frac{x_d}{z_d}z$. So, we have 
\[\sum_{i=1}^{d} x_i f_{e(v,i)} = \sum_{i=1}^{d-1} x^\dagger_i f_{e(v,i)} + \frac{x_d}{z_d}\sum_{i=1}^dz_if_{e(v,i)}.\]
The first sum on the right side is zero by (\ref{satisfied0}) since $\pi\left(x^\dagger\right)$ is contained in $X_0$. The second sum is equal to
\[\frac{x_d}{z_d}\sum_{i=1}^{d-1}z_if_{e(v,i)} + x_d f_{e(v,d)}\]
which is zero by \eqref{xstar}. 

Now, suppose that $v\in V\left(Q_{d-1}^1\right)$ and let $v'$ be the unique neighbour of $v$ in $V\left(Q_{d-1}^0\right)$. Given $x\in X$, we have 
\[\sum_{i=1}^{d} x_i f_{e(v,i)} = \sum_{i=1}^{d} x_i \left(f_{e(v',i)}^0\oplus f_{e(v,i)}^1\right) = \sum_{i=1}^{d} x_i \left(f_{e(v',i)}^0\oplus 0\right) + \sum_{i=1}^{d-1}x_i \left(0\oplus f_{e(v,i)}^1\right)\]
\[=\sum_{i=1}^{d}x_if_{e(v',i)} + \sum_{i=1}^{d-1}x_i \left(0\oplus f_{e(v,i)}^1\right).\]
The first sum on the right side is zero by the result of the previous paragraph (since $v'\in V\left(Q_{d-1}^0\right)$). The second sum is equal to zero by (\ref{satisfied1}) since $\pi(x)\in X_1$. Therefore, (\ref{satis}) holds. This completes the proof of the lemma.
\end{proof}

\section{General Grids}
\label{generalSec}

Our objective in this section is to determine the weak saturation number of $S_{r+1}$ in $\prod_{i=1}^d[a_i]$ in full generality. We express this weak saturation number in terms of the following recurrence relation. 

\begin{defn}
\label{recursion}
Let $d$ and $r$ be integers such that $0\leq r\leq 2d$ and let $a_1,\ldots,a_d \ge 2$. Define $w_r(a_1,\ldots,a_d)$ as follows:
\begin{itemize}[align=parleft,leftmargin=104pt,labelwidth=\widthof{(Hypercube Case 2)}]
\item[($r=0$ Case)] $w_r(a_1,\ldots,a_d) = 0  \text{ if } r=0$;
\item[($r=2d$ Case)] \( w_r(a_1,\ldots,a_d) = \displaystyle  \sum_{j=1}^d(a_j-1)\prod_{i\neq j}a_i  \text{ if }r=2d; \)
\item[(Hypercube Case 1)] \( w_r(a_1,\ldots,a_d) = \displaystyle d2^{d-1} \text{ if } {a_1=\dots=a_d=2\text{ and }d+1\leq r\leq 2d-1}; \)
\item[(Hypercube Case 2)] \( w_r(a_1,\ldots,a_d) = \displaystyle r2^{r-1} + \sum_{j=1}^{r-2} \binom{d-j-1}{r -j}j 2^{j-1} \text{ if } a_1=\dots=a_d=2\) and \(1\leq r\leq d; \text{ and } \)
\item[(Inductive Case)] \( w_r(a_1,\ldots,a_d) =\displaystyle {w_r(a_1,\dots,a_{i-1},a_i-1,a_{i+1},\dots a_d)} \\+ {w_{r-1}(a_1,\dots,a_{i-1},a_{i+1},\dots a_d)} \\+ {\sum_{\substack{S\subseteq[d]\setminus\{i\}\\ |S|\geq 2d-r}}2^{|S|}\prod_{j\notin S}(a_j-2)}\text{ if }1\leq r\leq 2d-1\text{ and }a_i\geq3.\)
\end{itemize}
\end{defn}

We prove the following.

\begin{thm}\label{wsatgrid}
For $0 \le r \le 2d$ and $a_1,\dots,a_d\geq2$, we have 
$$\wsat\left(\prod_{i=1}^d[a_i], S_{r+1}\right) = w_r(a_1,\ldots,a_d).$$
\end{thm}

Before presenting the proof, let us pause for a few remarks. Note that we do not need to provide a separate proof that the recurrence for $w_r(a_1,\dots,a_d)$ given above has a (unique) solution since this will be implied once we have shown  that $w_r(a_1,\dots,a_d)$ is equal to $\wsat\left(\prod_{i=1}^d[a_i], S_{r+1}\right)$. Also, in the case $d\geq r$, one can show that the expression in Theorem~\ref{genwsat} satisfies the recurrence in Definition~\ref{recursion} and so Theorem~\ref{genwsat} follows from Theorem~\ref{wsatgrid}. In what follows, let $a_1,\dots,a_{d}\geq2$ and define $G:=\prod_{i=1}^d[a_i]$. 

Note that a vertex $v$ of $G$ may be incident to either one or two edges in direction $i\in[d]$ depending on whether or not $v_i\in\left\{1,a_i\right\}$. With this in mind, we define a labelling of the edges of $G$. 

\begin{defn}
Say that an edge $e=uv\in E(G)$ in direction $i\in [d]$ is \emph{odd} if $\min\left\{u_i,v_i\right\}$ is odd and \emph{even} otherwise. We label $e$ by $e(v,2i-1)$ if $e$ is odd and $e(v,2i)$ if $e$ is even.
\end{defn}

Note that each edge of $G$ receives two labels, one for each of its endpoints. See Figure \ref{labels} for an explicit example of how we label the edges and define $I_v^G$.

\begin{defn}
For $v\in V(G)$, define $I^G_v:=\left\{j\in [2d]: e(v,j)\in E(G)\right\}$.
\end{defn}

\begin{figure}[htbp]
\centering 
\begin{tikzpicture}[main node/.style={circle,fill=white!20,draw}, scale=0.6, every node/.style={scale=0.6}]
\node [main node](a1) at (0,0) {$(1,1)$} ;
\node [main node](a2) at (0,4) {$(1,2)$} ;
\node [main node](a3) at (0,8) {$(1,3)$} ;
\node [main node](b1) at (4,0) {$(2,1)$} ;
\node [main node](b2) at (4,4) {$(2,2)$} ;
\node [main node](b3) at (4,8) {$(2,3)$};
\node [main node](c1) at (8,0) {$(3,1)$} ;
\node [main node](c2) at (8,4) {$(3,2)$} ;
\node [main node](c3) at (8,8) {$(3,3)$} ;
\draw 
(a1) edge[-]  node[midway,above,sloped]{{$e((1,2),1)$}} node[midway,below]{{$e((1,1),1)$}} (b1) 
(a1) edge[-] node[midway,below,sloped]{{$e((1,1),3)$}} node[midway,above,sloped]{{$e((1,2),3)$}}(a2) 
(b2) edge[-] node[midway,below,sloped, anchor = north]{{$e((2,1),3)$}} node[midway,above,sloped]{{$e((2,2),3)$}} (b1)

(b1) edge[-] node[midway,below]{{$e((2,1),2)$}} node[midway,above]{{$e((3,1),2)$}}(c1)
(c2) edge[-] node[midway,above,sloped]{{$e((3,2),3)$}} node[midway,below,sloped]{{$e((3,1),3)$}}(c1) 
(a2) edge[-] node[midway,below,sloped]{{$e((1,2),4)$}} node[midway,above,sloped]{{$e((1,3),4)$}} (a3) 
(b3) edge[-] node[midway,below,sloped, anchor = north]{{$e((2,2),4)$}} node[midway,above,sloped]{{$e((2,3),4)$}} (b2) 
(c3) edge[-] node[midway,below,sloped, anchor = north]{{$e((3,2),4)$}} node[midway,above,sloped]{{$e((3,3),4)$}}(c2) 
(a2) edge[-] node[midway,below]{{$e((1,2),1)$}} node[midway,above]{{$e((2,2),1)$}} (b2) 
(b2) edge[-] node[midway,below]{{$e((2,2),2)$}} node[midway,above]{{$e((3,2),2)$}} (c2) 
(a3) edge[-] node[midway,below]{{$e((1,3),1)$}} node[midway,above]{{$e((2,3),1)$}} (b3) 
(b3) edge[-] node[midway,below]{{$e((2,3),2)$}} node[midway,above]{{$e((3,3),2)$}}(c3) ;

\end{tikzpicture}
\caption{The edge labels where $G = \prod_{i=1}^2 [3]$. The lower label of each pair corresponds to the vertex on the left of the pair. We have $I_{(1,1)}^{G} = \{1,3\}$, $I_{(2,2)}^G = \{1,2,3,4\}$ and $I_{(3,3)}^G = \{2,4\}$, to give three examples of $I_v^G$.} 
\label{labels}
\end{figure}
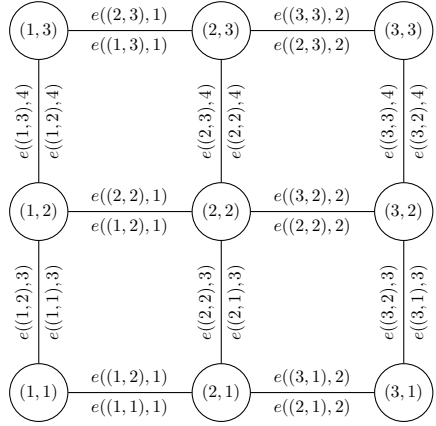

We are now in position to prove Theorem~\ref{wsatgrid}.  As with the proof of Theorem~\ref{wsatcube}, we state a lemma from which we deduce Theorem~\ref{wsatgrid}, and then we prove the lemma.

\begin{lem}
\label{theGridProp}
Let $X$ be a subspace of $\mathbb{R}^{2d}$ of dimension $2d-r$ such that $|\supp(x)|\geq r+1$ for every $x\in X\setminus\{0\}$. There is a spanning subgraph $F$ of $G$ and a collection $\left\{f_e: e\in E\left(G\right)\right\}\subseteq \mathbb{R}^{w_r(a_1,\dots,a_d)}$ such that
\begin{enumerate}
\renewcommand{\theenumi}{G\arabic{enumi}}
\renewcommand{\labelenumi}{(\theenumi)}
{\setlength\itemindent{4pt}\item \label{Gwsat} $F$ is weakly $\left(G,S_{r+1}\right)$-saturated and $|E(F)|=w_r(a_1,\dots,a_d)$,}
{\setlength\itemindent{4pt}\item \label{Gsatisfied} $\sum_{i=1}^{2d} x_i f_{e(v,i)} = 0$ for every $v\in V\left(G\right)$ and $x\in X$ such that $\supp(x)\subseteq I^G_v$, and}
{\setlength\itemindent{4pt}\item \label{GbigDim} $\vspan\left\{f_e: e\in E\left(G\right)\right\}=\mathbb{R}^{w_r(a_1,\dots,a_d)}$.}
\end{enumerate}
\end{lem}

\begin{proof}[Proof of Theorem~\ref{wsatgrid}.]
First observe that the existence of a graph $F$ satisfying (\ref{Gwsat}) implies $\wsat(G, S_{r+1}) \le w_r(a_1,\dots,a_d)$. To obtain a matching lower bound, we apply Lemma~\ref{linalg} as we did in the hypercube case. The edge sets of copies of $S_{r+1}$ in $G$ are the sets of the form $\{e(v,i): i \in T\}$, where $v \in V(G)$ and $T$ is a subset of $I^G_v$ of cardinality $r+1$. By applying Lemma~\ref{Tsupp} together with (\ref{Gsatisfied}), we see that the conditions of Lemma~\ref{linalg} are satisfied. Thus by (\ref{GbigDim}), $\wsat(G,S_{r+1}) \geq w_r(a_1,\dots,a_d)$.
\end{proof}

We now prove Lemma~\ref{theGridProp} in a similar fashion to the proof of Lemma~\ref{theProp}. We remark that the cases in the proof of this lemma correspond precisely to the cases of Definition~\ref{recursion}. 

\begin{proof}[Proof of Lemma~\ref{theGridProp}]
We proceed by induction on $|V(G)|$. We begin with the boundary cases.

\begin{case3}
$r = 0$.
\end{case3} 

In this case, $S_{r+1}$ is isomorphic to $K_2$. Also, $w_r(a_1,\dots,a_d)=0$ and $X=\mathbb{R}^{2d}$. We let $F$ be a spanning subgraph of $G$ with no edges and set $f_e:=0$ for every $e\in Q_d$. Properties (\ref{Gwsat}), (\ref{Gsatisfied}) and (\ref{GbigDim}) are satisfied trivially. 

\begin{case3}
$r=2d \geq 2$.
\end{case3} 

In this case, $w_r(a_1,\dots,a_d)=|E(G)|$ and $X=\{0\}$. We define $F:=G$ and let $\{f_e:e\in E(G)\}$ be a basis for $\mathbb{R}^{w_r(a_1,\dots,a_d)}$. Clearly (\ref{Gwsat}), (\ref{Gsatisfied}) and (\ref{GbigDim}) are satisfied.

\begin{case3}
$a_1 = \ldots = a_d=2$ and $d+1\leq r\leq 2d-1$.
\end{case3}

In this case, $G$ is isomorphic to $Q_d$ and we again have $w_r(a_1,\dots,a_d)=|E(G)|$. We define $F:=G$ and let $\{f_e:e\in E(G)\}$ be a basis for $\mathbb{R}^{w_r(a_1,\dots,a_d)}$. Clearly (\ref{Gwsat}), (\ref{Gsatisfied}) and (\ref{GbigDim}) are satisfied.

\begin{case3}
$a_1 = \ldots = a_d=2$ and $1\leq r\leq d$.
\end{case3} 

We again have that $G$ is isomorphic to $Q_d$. Also, note that every edge of $G$ is odd. We let $X'$ be the subspace of $X$ consisting of all vectors $x$ of $X$ such that every element of $\supp(x)$ is odd. 

We claim that $X'$ has dimension $d-r$. First, $X'$ is contained in the vector space consisting of all vectors in $\mathbb{R}^{2d}$ whose support contains only odd numbers, which is of course isomorphic to $\mathbb{R}^d$. Also, the support of every non-zero vector in $X'$ has size at least $r+1$ which, by Lemma~\ref{Tsupp}, implies that $X'$ has dimension at most $d-r$. Applying  Lemma~\ref{Tsupp} to $X$ we see that, for every odd number $j$ such that $2r+1\leq j\leq 2d-1$ there exists a vector $x_j\in X$ with $\supp\left(x_j\right) = \{1,3,\dots, 2r-1,j\}$. These vectors are linearly independent and contained in $X'$, and so $X'$ has dimension exactly $d-r$. 

Now, since $X'$ has dimension $d-r$, every $x\in X'$ has $|\supp(x)|\geq r+1$ and every edge in $G$ is odd,  we get that the graph $F$ and the vectors $\{f_e: e\in E(Q_d)\}$ exist by Lemma~\ref{theProp}. 

\begin{case3}
$1\leq r\leq 2d-1$ and $a_i\geq3$ for some $i\in [d]$. 
\end{case3} 

Without loss of generality, assume that $a_d\geq3$. Define
\[G_1:=\prod_{i=1}^{d-1}[a_i]\times [a_d-1],\text{ and}\]
\[G_2:=G\setminus G_1.\] 
Observe that every vertex of $G_2$ has a unique neighbour in $V(G_1)$. The edges with one endpoint in $G_1$ and the other in $G_2$ will play a particular role in the proof. We define
$$\tau := \begin{cases}      
    2d-1 & a_d - 1 \text{ is odd},\\
    2d & a_d - 1 \text{ is even},
\end{cases}$$
and we write $\bar{\tau}$ for the unique element of $\{2d-1,2d\}\backslash \{\tau\}$. Observe that for $v \in V(G_2)$, we have that $\bar{\tau} \notin I^G_v$, and that $I^{G_2}_v = I^G_v\setminus\{\tau\}$. On the other hand, if $v\in V(G_1)$, then
\[I_v^{G_1} = \begin{cases}	I_v^G\setminus \{\tau\} & \text{if }v_d= a_d-1,\\
										I_v^G 	& \text{otherwise}.\end{cases}\]
									
Before moving on, we need to make yet another definition. Define
\[Y := \left\{v \in V(G_1): v_d = a_d -1 \text{ and } d_{G_1}(v) < r\right\}.\]
Let us count the elements of $Y$. Every vertex of $G_1$ has either one or two neighbours in each direction $i\in [d]$. Thus, the degree of a vertex $v$ in $G_1$ is equal to $2d$ minus the number of directions in which $v$ has only one neighbour. So, if $d_{G_1}(v) <r$, then there must be at least $2d-r+1$ directions in which $v$ has only one neighbour. Observe that $v$ has a unique neighbour in some direction $i\in [d]$ if and only if $v_i\in\{1,a_i\}$. So, $|Y|$ is equal to the number of ways to choose a set $S\subseteq [d-1]$ of size $2d-r$ and to choose a vector $v$ of $G_1$ such that $v_i\in\{1,a_i\}$ for $i\in S$, $v_i\in \{2,\dots, a_i-1\}$ for $i\in [d-1]\setminus S$ and $v_d=a_d-1$. Thus, 
\begin{equation}\label{yCard}|Y|= {\sum_{\substack{S\subseteq[d-1]\\ |S|\geq 2d-r}}2^{|S|}\prod_{j\notin S}(a_j-2)}.\end{equation}
For brevity we write $y:=|Y|$ and
\[w_1:=w_r(a_1,\dots,a_{d-1},a_d - 1),\]
\[w_2:=w_{r-1}(a_1,\dots,a_{d-1}).\]
Note that, by induction, $w_1=\wsat\left(G_1,S_{r+1}\right)$ and $w_2=\wsat\left(G_2,S_r\right)$.  Next, we construct a graph $F$ satisfying (\ref{Gwsat}). Define $F$ to be a spanning subgraph of $G$ such that

\begin{itemize}
\item the subgraph $F_1$ of $F$ induced by $V(G_1)$ is a weakly $(G_1,S_{r+1})$-saturated graph of minimum size,
\item the subgraph $F_2$ of $F$ induced by $V(G_2)$ is a weakly $(G_2,S_{r})$-saturated graph of minimum size, and
\item an edge $e$ from $V(G_1)$ to $V(G_2)$ is contained in $F$ if and only if $e$ is of the form $e(v,\tau)$ for $v\in Y$. 
\end{itemize}

By \eqref{yCard} and Definition~\ref{recursion}, we see that 
\[|E(F)| = w_1+w_2+y = w_r(a_1,\dots,a_d),\]
as required.  To see that $F$ is weakly $(G,S_{r+1})$-saturated, we add the edges of $E(G)\backslash E(F)$ to $F$ in three stages. First, by definition of $F_1$, we can add the edges that are not present in $E(F_1)$ in such a way that every added edge completes a copy of $S_{r+1}$ in $G_1$. Next, we can add the edges of the form $e(v,\tau)$, where $v\notin Y$ and $v_d = a_d -1$, in any order. By definition of $Y$, we see that every such $v$ has at least $r$ neighbours in $G_1$. As every edge in $E(G_1)$ has already been added, the addition of $e(v,\tau)$ completes a copy of $S_{r+1}$ in $G$. Finally, we add the edges of $G_2$ that are not present in $F_2$ in such a way that each added edge completes a copy of $S_r$ in $G_2$. Every such edge completes a copy of $S_{r+1}$ in $G$ since every vertex in $G_2$ has a neighbour in $G_1$ and every edge between $G_1$ and $G_2$ is already present. Thus, (\ref{Gwsat}) holds.

It remains to find a collection $\{f_e: e \in E(G)\}$ satisfying (\ref{Gsatisfied}) and (\ref{GbigDim}). Let $\pi:X \rightarrow \mathbb{R}^{2d-2}$ be the projection defined by $\pi:(x_1, \ldots, x_{2d}) \mapsto (x_1,\ldots, x_{2d -2})$. Let $z$ be a fixed vector of $X$ such that $\bar{\tau}\in \supp\left(z\right)$ and define $T_{z}:X\to X$ by
\[T_{z}(x):= x-\frac{x_{\bar{\tau}}}{z_{\bar{\tau}}}z.\]
Define $X_1:=X$ and $X_2:= \pi\left(T_{z}(X)\right)$. 

Clearly $X_2$ is a subspace of $\mathbb{R}^{2(d-1)}$. In order to apply the inductive hypothesis on $X_2$ and $G_2$, we need that $X_2$ has dimension $2d-r-1 = 2(d-1) - (r-1)$ and that every $x\in X_2$ has $|\supp(x)|\geq r$. The kernel of $T_z$ is precisely the span of $\{z\}$, and so $T_z(X)$ is a subspace of $X$ of dimension $2d-r-1$. In particular, since $T_z(X)\subseteq X$, every non-zero $x\in T_z(X)$ has $|\supp(x)|\geq r+1$. By definition of $T_z$, every $x\in T_z(X)$ satisfies $x_{\bar{\tau}}=0$ and so $\left|\supp(x)\cap \{2d-1,2d\}\right|\leq 1$. In particular, since $r+1\geq2$, no non-zero vector of $T_z(X)$ is mapped by $\pi$ to the zero vector, and so we get that $X_2$ has dimension $2d-r-1$. Applying the fact that every $x\in T_z(X)$ satisfies $\left|\supp(x)\cap \{2d-1,2d\}\right|\leq 1$ once again, we get that the support of every vector in $X_2$ has size at least $r$.

Now, by applying the inductive hypothesis to both $G_1$ and $G_2$, we can find collections $\{f^1_e: e \in E(G_1)\}$ in $\mathbb{R}^{w_1}$ and $\{f^2_e: e \in E(G_2)\}$ in $\mathbb{R}^{w_2}$ such that
\begin{enumerate}
\renewcommand{\theenumi}{G2.\arabic{enumi}}
\renewcommand{\labelenumi}{(\theenumi)}
{\setlength\itemindent{10pt}\item\label{Gsatisfied0} $\sum_{i=1}^{2d} x_i f_{e(v,i)}^1 = 0$ for every $v\in V\left(G_1\right)$ and $x\in X_1$ with $\supp(x) \subseteq I^{G_1}_v$,}
{\setlength\itemindent{10pt}\item\label{Gsatisfied1} $\sum_{i=1}^{2d-2} x_i f_{e(v,i)}^2 = 0$ for every $v\in V\left(G_2\right)$ and $x\in X_2$ with $\supp(x) \subseteq I^{G_2}_v$,}
\renewcommand{\theenumi}{G3.\arabic{enumi}}
\renewcommand{\labelenumi}{(\theenumi)}
\addtocounter{enumi}{-2}
{\setlength\itemindent{10pt}\item\label{GbigDim0} $\vspan\left\{f_e^1: e\in E\left(G_1\right)\right\}=\mathbb{R}^{w_1}$, and}
{\setlength\itemindent{10pt}\item\label{GbigDim1} $\vspan\left\{f_e^2: e\in E\left(G_2\right)\right\}=\mathbb{R}^{w_2}$. }
\end{enumerate}

Using this, we will now construct a collection $\{f_e: e \in E(G)\} \subseteq \mathbb{R}^{w_1} \oplus \mathbb{R}^{w_2} \oplus \mathbb{R}^{y} \simeq \mathbb{R}^{w_r\left(a_1,\dots,a_d\right)}$ in four steps. First, for $e \in E(G_1)$, we define
$$f_e := f_e^1 \oplus 0 \oplus 0.$$ 
Let $\{f_y^3: y \in Y\}$ be a basis of $\mathbb{R}^{y}$. Next, we consider edges $e = uv$, where $v \in V(G_1)$, and $u \in V(G_2)$. If $v$ is in $Y$, then we let
$$f_e := 0 \oplus 0 \oplus f^3_v.$$ 
If $v$ is not in $Y$, then let $z^v\in X$ be a vector such that $\supp(z^v)\subseteq I_v^G$ and $\tau\in \supp\left(z^v\right)$, which exists by Lemma~\ref{support}. Define
\begin{equation}
\label{xstar2}
f_e := -\frac{1}{z^v_\tau}\sum_{i\in [2d]\setminus\{\tau\}}z^v_i f_{e(v,i)}.
\end{equation}
Finally if $e =uv \in E(G_2)$, then let $e' = u'v'$ where $u'v'$ are the unique neighbours of $u$ and $v$ in $V(G_1)$ and define
$$f_e := f^1_{e'}\oplus f^2_{e} \oplus 0.$$
It is clear from (\ref{GbigDim0}), (\ref{GbigDim1}) and the construction of $f_e$, that the dimension of $\vspan \{f_e:e \in E(G)\}$ is $w_1 + w_2 + y = w_r(a_1,\dots,a_d)$. Thus (\ref{GbigDim}) is satisfied.

It remains to show that (\ref{Gsatisfied}) holds. Firstly, suppose $v \in V(G_1)$ and let $x \in X$ be such that $\supp(x) \subseteq I^G_v$. If $v_d < a_d - 1$, then  $\sum_{i=1}^{2d} x_i f_{e(v,i)} = 0$ by (\ref{Gsatisfied0}). If $v_d = a_d - 1$ and $v \in Y$, then, by definition of $Y$, we have $|I^G_v| \le r$. Thus, by our hypothesis on $X$, the only vector $x\in X$ with $\supp(x)\subseteq I^G_v$ is the zero vector and so we are done. Now suppose that $v\notin Y$ and that $v_d = a_d -1$. Define 
\[x^{\dagger} := x - \frac{x_\tau}{z^v_\tau}z^v\]
and note that $x=x^\dagger + \frac{x_\tau}{z^v_\tau}z^v$. We have,
\begin{equation}\label{G2sat}
\sum_{i=1}^{2d}x_if_{e(v,i)} = \sum_{i=1}^{2d} x^{\dagger}_i f_{e(v,i)} + \frac{x_{\tau}}{z^v_{\tau}}\sum_{i=1}^{2d}z_i^vf_{e(v,i)}.
\end{equation}
Note that $\tau \notin \supp(x^{\dagger})$ and thus $\supp(x^\dagger) \subseteq I^{G}_v\setminus\{\tau\} = I^{G_1}_v$. Thus, the first sum on the right side of \eqref{G2sat} is equal to
\[\sum_{i\in [2d]\setminus\{\tau\}} x^{\dagger}_i \left(f_{e(v,i)}^1 \oplus 0\oplus 0\right)\]
which is zero by (\ref{Gsatisfied0}). The second sum is equal to 
\[\frac{x_\tau}{z_\tau^v}\sum_{i\in [2d]\setminus\{\tau\}} z_i^v f_{e(v,i)} + x_\tau f_{e(v,\tau)}\]
which is zero by \eqref{xstar2}. 

Finally, consider $v \in V(G_2)$. Let $v'$ be the unique neighbour of $v$ in $V(G_2)$. Given $x \in X$, with $\supp(x) \subseteq I^G_{v}$ we have
$$\sum_{i=1}^{2d} x_i f_{e(v,i)} = \sum_{i=1}^{2d}x_i f_{e(v',i)} + \sum_{i=1}^{2d - 2}x_i \left(0 \oplus f^2_{e(v,i)} \oplus 0\right).$$
Since $I_v^{G}\subseteq I_{v'}^{G}$ we have that $\supp(x)\subseteq I_{v'}^G$ and so the first sum on the right side is equal to zero by the result of the previous paragraph. The second sum on the right side is zero by (\ref{Gsatisfied1}), which is applicable as $\bar{\tau} \notin I_v^G\supseteq \supp(x)$, and so $x \in T_{z}(X)$. This completes the proof of the lemma.
\end{proof}

\section{Upper Bound Constructions}
\label{upperSec}

In this section, we prove a recursive upper bound on $m(Q_d,r)$ for general $d\geq r\geq 1$ and then apply it to obtain an exact expression for $m(Q_d,3)$. 

\begin{lem}
\label{generalRUpper}
For $d\geq r\geq1$,
\[m\left(Q_d,r\right) \leq  m\left(Q_{d-r},r\right) + (r-1)m\left(Q_{d-r},r-1\right) + \sum_{j=1}^{\left\lceil r/2\right\rceil-1}\binom{r}{2j+1}m\left(Q_{d-r},r-2j\right).\]
\end{lem}

\begin{proof}
Let $d\geq r$ be fixed positive integers. For $1\leq t\leq r$, let $B_t$ be a subset of $V\left(Q_{d-r}\right)$ of cardinality $m\left(Q_{d-r},t\right)$ which percolates with respect to the $t$-neighbour bootstrap percolation process in $Q_{d-r}$. 

Given $x\in V(Q_d)$, let $[x]_r$ and $[x]_{d-r}$ denote the vectors obtained by restricting $x$ to its first $r$ coordinates and last $d-r$ coordinates, respectively. We partition $\{0,1\}^r$ into $r+1$ sets $L_0,\dots,L_r$ such that $L_i$ consists of the vectors whose coordinate sum is equal to $i$. We construct a percolating set $A_0$ in $Q_d$. Given $x\in V(Q_d)$, we include $x$ in $A_0$ if one of the following holds:
\begin{itemize}
\item $[x]_r\in L_1$ and either
\begin{itemize}
\item $[x]_r = (1,0,\dots,0)$ and $[x]_{d-r}\in B_r$.
\item $[x]_r \neq (1,0,\dots,0)$ and $[x]_{d-r}\in B_{r-1}$.
\end{itemize}
\item $[x]_r\in L_{2j+1}$ for some $1\leq j\leq \left\lceil r/2\right\rceil-1$ and $[x]_{d-r}\in B_{r-2j}$. 
\end{itemize}
It is clear that 
\[|A_0| = m\left(Q_{d-r},r\right) + (r-1)m\left(Q_{d-r},r-1\right) + \sum_{j=1}^{\left\lceil r/2\right\rceil-1}\binom{r}{2j+1}m\left(Q_{d-r},r-2j\right)\]
by construction. We will be done if we can show that $A_0$ percolates with respect to the $r$-neighbour bootstrap percolation process. 

We begin by showing that every vertex $x$ with $[x]_r\in L_0\cup L_1$ is eventually infected. First, we can infect every vertex $x$ such that $[x]_r = (1,0,\dots,0)$, one by one in some order, by definition of $B_r$. Next, consider a vertex $x$ such that $[x]_r\in L_0$ and $[x]_{d-r}\in B_{r-1}$. Then $x$ has $r-1$ neighbours $z\in A_0$ such that $[z]_r\neq (1,0,\dots,0)$, by construction, and one infected neighbour $y$ such that $[y]_r=(1,0,\dots,0)$. Thus, every such $x$ becomes infected. Now, by definition of $B_{r-1}$, the remaining vertices $x$ such that $[x]_r\in L_0$ can be infected since every such vertex has an infected neighbour $y$ such that $[y]_r=(1,0,\dots,0)$. Finally, each vertex $x$ such that $x\neq (1,0,\dots,0)$ and $[x]_r\in L_1$ becomes infected using the definition of $B_{r-1}$ and the fact that every vertex $y$ with $[y]_r\in L_0$ is already infected. 

Now, suppose that, for some $1\leq j\leq \left\lceil r/2\right\rceil-1$ every vertex $x$ such that $[x]_r \in L_0\cup\cdots\cup L_{2j-1}$ is already infected. We show that every vertex $x$ with $[x]_r\in L_{2j}\cup L_{2j+1}$ is eventually infected. First, consider a vertex $x$ with $[x]_r\in L_{2j}$ and $[x]_{d-r}\in B_{r-2j}$. Such a vertex has $2j$ infected neighbours $y$ such that $[y]_r\in L_{2j-1}$ and $r-2j$ neighbours $z$ such that $[z]_{r}\in L_{2j+1}\cap A_0$. Therefore, every such $x$ becomes infected. Now, by definition of $B_{r-2j}$, the remaining vertices $x$ such that $[x]_r\in L_{2j}$ can be infected since every such vertex has $2j$ infected neighbours $y$ such that $[y]_r\in L_{2j-1}$. Finally, each vertex $x$ such that $[x]_r\in L_{2j+1}$ becomes infected using the definition of $B_{r-2j}$ and the fact that every vertex $y$ with $[y]_r\in L_{2j-1}$ is already infected. 

Finally, if $r$ is even, then we need to show that every vertex of $L_r$ becomes infected. Every such vertex has precisely $r$ neighbours in $L_{r-1}$. Thus, given that every vertex of $L_{r-1}$ is infected, $x$ becomes infected as well. This completes the proof.
\end{proof}

\begin{figure}[ht]
\newcommand{\Depth}{2.1}
\newcommand{\Width}{1.3}
\tikzstyle{loosely dashed}=          [dash pattern=on 5pt off 3pt]
\begin{center}
\begin{tikzpicture}

\node[circle,draw,minimum width=22pt] (A) at (0,0){};
\node[circle,draw,minimum width=22pt] (B) at (-\Depth,\Width){$3$};
\node[circle,draw,minimum width=22pt] (C) at (0,2*\Width){};
\node[circle,draw,minimum width=22pt] (D) at (\Depth,\Width){$2$};
\node[circle,draw,minimum width=22pt] (E) at (0,\Width){$2$};
\node[circle,draw,minimum width=22pt] (F) at (-\Depth,2*\Width){};
\node[circle,draw,minimum width=22pt] (G) at (0,3*\Width){$1$};
\node[circle,draw,minimum width=22pt] (H) at (\Depth,2*\Width){};

\draw[very thick] (A) -- (B) -- (C) -- (D) -- (A);
\draw[very thick] (A) -- (E);
\draw[very thick] (B) -- (F);
\draw[very thick] (C) -- (G);
\draw[very thick] (D) -- (H);
\draw[very thick] (E) -- (F) -- (G) -- (H) -- (E);
\end{tikzpicture}
\end{center}

\caption{An illustration of the set $A_0$ constructed in the proof of Lemma~\ref{generalRUpper} in the case $r=3$. Each node represents a copy of $Q_{d-3}$. The set $A_0$ consists of a copy of $B_i$ on each node labelled $i\in\{1,2,3\}$.}
\label{construction}
\end{figure}
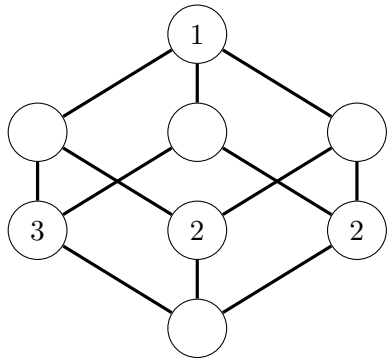

We remark that the recursion in Lemma~\ref{generalRUpper} gives a bound of the form $m\left(Q_d,r\right)\leq \frac{1+o(1)d^{r-1}}{r!}$ where the second order term is better than the one in \eqref{Steiner}. Next, we prove Theorem~\ref{r=3Upper}.

\begin{proof}[Proof of Theorem~\ref{r=3Upper}]
The lower bound follows from Theorem~\ref{hyper}. We prove the upper bound by induction on $d$. First, we settle the cases $d\in\{3,\dots,8\}$. For notational convenience, we associate each element $v$ of $\{0,1\}^d$ with the of subset of $[d]$ for which $v$ is the characteristic vector. Moreover, we identify each non-empty subset of $[d]$ with the concatenation of its elements (e.g. $\{1,3,7\}$ is written $137$). One can verify (by hand or by computer) that the set $A_0^d$, defined below, percolates with respect to the $3$-neighbour bootstrap percolation process in $Q_d$ and that it has the cardinality $\left\lceil\frac{d(d+3)}{6}\right\rceil+1$. 
\begin{align*}A_0^3 &:= \{1,2,3,123\},\\
A_0^4 &:= \left(A_0^3\setminus\{3\}\right)\cup\{134,4,234\},\\
A_0^5 &:= \left(A_0^4\setminus\{134\}\right)\cup\{135,245,12345\},\\
A_0^6 &:= \left(A_0^5\setminus\{135,245\}\right)\cup\{346,12356,456,23456\},\\
A_0^7 &:= \left(A_0^6\setminus\{346\}\right)\cup\{13457,24567,12367,1234567\},\\
A_0^8 &:= \left(A_0^7\setminus\{13457,24567\}\right)\cup\{34568,1234578,34678,25678,2345678\}.\end{align*}

Now, suppose $d\geq9$ and that the theorem holds for smaller values of $d$. If $d$ is odd, then we apply Lemma~\ref{generalRUpper} to obtain
\[m\left(Q_d,3\right)\leq m\left(Q_{d-3},3\right)+ 2m\left(Q_{d-3},2\right)+ m\left(Q_{d-3},1\right).\]
Clearly, $m\left(Q_{d-3},1\right)=1$ and it is easy to show that $m\left(Q_{d-3},2\right) \leq \frac{d-3}{2}+1$ (since $d-3$ is even). Therefore, by the inductive hypothesis,
\[m\left(Q_d,3\right)\leq \left\lceil\frac{(d-3)d}{6}\right\rceil+1 + 2\left(\frac{d-3}{2}+1\right)+1 = \left\lceil\frac{d(d+3)}{6}\right\rceil+1.\]

Now, suppose that $d\geq10$ is even. For $t\in\{1,2,3\}$, let $B_t$ be a subset of $V\left(Q_{d-6}\right)$ of cardinality $m\left(Q_{d-6},t\right)$ which percolates with respect to the $t$-neighbour bootstrap percolation process on $Q_{d-6}$ and let $A_0^6$ be as above. Given a vector $x\in V(Q_d)$, let $[x]_6$ be the restriction of $x$ to its first six coordinates and $[x]_{d-6}$ be the restriction of $x$ to its last $d-6$ coordinates. We define a subset $A_0$ of $V(Q_d)$. We include a vertex $x\in V(Q_d)$ in $A_0$ if $[x]_6\in A_0^6$ and one of the following holds:
\begin{itemize}
\item $[x]_6 = (0,0,1,1,0,1)$ and $[x]_{d-6}\in B_3$.
\item $[x]_6 \neq (0,0,1,1,0,1)$ and we have $x_5=1$ and $[x]_{d-6}\in B_2$.
\item $x_5=x_6=0$ and $[x]_{d-6}\in B_1$. 
\end{itemize}
The fact that $A_0$ percolates follows from arguments similar to those given in the proof of Lemma~\ref{generalRUpper}; we omit the details. By construction, 
\[|A_0| = m\left(Q_{d-6},3\right) + 4m\left(Q_{d-6},2\right) + 5m\left(Q_{d-6},1\right)\]
which equals
\[\left\lceil\frac{(d-6)(d-3)}{6}\right\rceil+1 + 4\left(\frac{d-6}{2}+1\right) + 5 = \left\lceil\frac{d(d+3)}{6}\right\rceil+1\]
by the inductive hypothesis. The result follows. 
\end{proof}

\section{Concluding Remarks}
\label{concl}

In this paper, we have determined the main asymptotics of $m\left(Q_d,r\right)$ for fixed $r$ and $d$ tending to infinity and obtained a sharper result for $r=3$. We wonder whether sharper asymptotics are possible for general $r$. 

\begin{ques}
For fixed $r\geq4$ and $d\to \infty$, does 
\[\frac{m\left(Q_d,r\right)-\frac{d^{r-1}}{r!}}{d^{r-2}}\]
converge? If so, what is the limit?
\end{ques}

As Theorem~\ref{r=3Upper} illustrates, it may be possible to obtain an exact expression for $m\left(Q_d,r\right)$ for some small fixed values of $r$. The first open case is the following.

\begin{prob}
Determine $m\left(Q_d,4\right)$ for all $d\geq4$. 
\end{prob}

Using a computer, we have determined that $m\left(Q_5,4\right) =14$, which is greater than the lower bound of $13$ implied by Theorem~\ref{hyper}. Thus, Theorem~\ref{hyper} is not tight for general $d$ and $r$. However, we wonder whether it could be tight when $r$ is fixed and $d$ is sufficiently large.

\begin{ques}
For fixed $r\geq4$, is it true that
\[m\left(Q_d,r\right)= 2^{r-1} + \left\lceil\sum_{j=1}^{r-1}\binom{d-j-1}{r-j}\frac{j2^{j-1}}{r}\right\rceil\]
provided that $d$ is sufficiently large?
\end{ques}

Another direction that one could take is to determine $\wsat(G,S_{r+1})$ for other graphs $G$. For example, one could consider the $d$-dimensional torus $\mathbb{Z}_n^d$. 

\begin{prob}
\label{torus}
Determine $\wsat\left(\mathbb{Z}_n^d,S_{r+1}\right)$ for all $n,d$ and $r$. 
\end{prob}

\begin{ack}
The authors would like to thank Eoin Long for encouraging us to work on this problem and for several stimulating discussions during the 18th Midrasha Mathematicae at the Institute of Advanced Studies in Jerusalem in 2015. We would also like to thank Micha{\l} Przykucki and Alex Scott for several enlightening discussions. In particular, we are grateful to the latter for bringing \eqref{wsatperc} to our attention and for helpful comments regarding the presentation of this paper. We also thank Hamed Hatami, Yingjie Qian, Oliver Riordan, Andrew Thomason and an anonymous referee for comments which have helped us to improve the presentation of the paper. 
\end{ack}

\begin{noteAdded}
Following the submission of this paper, Hambardzumyan, Hatami and Qian~\cite{McGill} discovered a new (and shorter) proof of our main result using the so called ``polynomial method.'' They have also applied their method to obtain a full solution to Problem~\ref{torus}. 
\end{noteAdded}

\begin{noteAdded}
After submitting this paper, the authors discovered that weak saturation of the star $S_3$ in $[n]^2$ had been studied in a 1984 paper of physicists Lenormand and Zarcone~\cite{LZ} as a ``bond percolation'' variant of bootstrap percolation. In~\cite{LZ}, Lenormand and Zarcone estimated (using large simulations) the critical probability for the event that a random spanning subgraph of $[n]^2$ with edge probability $p$ is weakly $([n]^2,S_3)$-saturated. This was proposed as a model of the spreading of liquid in a network of capillaries. It seems to be an interesting problem to obtain precise asymptotics for this critical probability for large $n$. This can be seen as a bond percolation analogue to the result of Holroyd~\cite{Holroyd} on the $2$-neighbour bootstrap percolation process in $[n]^2$. 
\end{noteAdded}

\bibliographystyle{plain}

\appendix

\section*{Appendix: An Explicit Linear Algebraic Construction}

Given integers $k$ and $\ell$ with $k\geq \ell\geq0$, we construct an explicit subspace $X$ of $\mathbb{R}^k$ of dimension $k-\ell$ such that $|\supp(x)|\geq \ell+1$ for every $x\in X\setminus\{0\}$. This can be seen as an alternative proof of Lemma~\ref{support}. 

Let $\alpha_1,\dots,\alpha_k$ be arbitrary distinct real numbers. Define $f$ to be a map from the space of polynomials $p$ of degree at most $k-\ell-1$ in a single variable over $\mathbb{R}$ to $\mathbb{R}^k$ defined by
\[f:p\mapsto (p(\alpha_1),\dots, p(\alpha_k)).\]
Define $X$ to be the range of $f$. Clearly $X$ is a vector space. Every polynomial of degree at most $k-\ell-1$ has at most $k-\ell-1$ distinct roots. This implies that every non-zero $x\in X$ must have $|\supp(x)|\geq \ell+1$. It also implies that the kernel of $f$ is $\{0\}$. So, the dimension of $X$ is equal to the dimension of the space of polynomials of degree at most $k-\ell-1$, which is of course $k-\ell$. We thank Hamed Hatami for bringing this family of constructions to our attention. 

Let us exhibit a particular basis of the space $X$. Note that by considering the polynomials $p_j(x):=x^j$ for $0\leq j\leq k-\ell-1$, we see that $X$ contains the vectors
\[(1,1,\dots,1)\]
\[(\alpha_1,\alpha_2,\dots,\alpha_k)\]
\[(\alpha_1^2,\alpha_2^2,\dots, \alpha_k^2)\]
\[\vdots\]
\[(\alpha_1^{k-\ell-1},\alpha_2^{k-\ell-1},\dots, \alpha_k^{k-\ell-1}).\]
The $(k-\ell)\times (k-\ell)$ matrix whose row vectors are the projections of the above vectors onto the first $k-\ell$ coordinates is known as a \emph{Vandermonde matrix}. The determinant of this matrix is well known to be $\prod_{1\leq i<j\leq k-\ell}\left(\alpha_i-\alpha_j\right)$, which is non-zero since $\alpha_1,\dots,\alpha_{k-\ell}$ are distinct. Therefore, the vectors in the above list are linearly independent and so they form a basis for $X$.
\end{document}